\documentclass[12pt]{amsart}
\usepackage{amsmath,amsthm,amsfonts,amssymb,amscd}
\usepackage[dvips]{graphicx}
\usepackage{psfrag}

\theoremstyle{plain}
\newtheorem{thm}{Theorem}
\newtheorem{prop}[thm]{Proposition}

\newtheorem{lemma}[thm]{Lemma}

\newtheorem{maintheorem}{Theorem}
\newtheorem{claim}[thm]{Claim}

\newcommand{\re}{{\mathbb R}}

\newcommand{\dist}{\operatorname{dist}}

\title[Stochastic stability of sectional-Anosov flows]{Stochastic stability of sectional-Anosov flows}
\author{R.J. Metzger}
\address{Instituto de Matem\'atica y Ciencias Afines \, IMCA\\Calle Los Bi\'ologos 245 \\ Urb. San C\'esar La Molina, Lima 12\\Lima, Per\'u.}
\email{metzger@imca.edu.pe,roger@impa.br}

\author{C.A. Morales}
\address{Instituto de Matem\'atica\\ Universidade Federal do Rio de Janeiro,
P. O. Box 68530, 21945-970 Rio de Janeiro, Brazil.}
\email{morales@impa.br}


\thanks{CAM was partially supported by CNPq, FAPERJ and
PRONEX/DYN-SYS. from Brazil.}

\begin{document}
\begin{abstract}
A {\em sectional-Anosov flow} is a vector field on a compact manifold
inwardly transverse to the boundary such that the maximal invariant set is sectional-hyperbolic
(in the sense of \cite{mm}).
We prove that any $C^2$ transitive sectional-Anosov flow
has a unique SRB measure which is stochastically stable under small random perturbations.
\end{abstract}

\maketitle

\section{Introduction}

\noindent
We shall study the {\em SRB measures}
were discovered by Sinai, Bowen and Ruelle in the 70's.
More precisely, we will concerned with continuous-time systems, i.e., vector fields and their corresponding flows.
The motivation comes from flow's counterpart of this discovering \cite{br} implying
that such measures do exist for any $C^2$ Anosov flow on a compact manifold.
Naturally, this conduces to the following question: Can the existence SRB measures be proved for dynamical systems beyond the Anosov ones?
An important case is that of the {\em sectional-Anosov flows} defined in \cite{m}.
These flows extend the Anosov ones to include important examples
like the {\em geometric} and {\em multidimensional Lorenz attractors} \cite{abs}, \cite{gw}, \cite{bpv}
and, specifically, Lorenz's polynomial flow \cite{l}, \cite{t}.
We therefore ask if, likewise Anosov's, every $C^2$ sectional-Anosov flow on a compact manifold
carries a SRB measure.
Positive answer for the geometric Lorenz attractor is nowadays folcklore.
For general $C^2$ sectional-Anosov flows on compact $3$-manifolds the answer is positive (with unicity)
in the transitive case. This can be deduced from \cite{col}
(with the assumption that the periodic points are dense in the maximal invariant set)
and from \cite{appv} (without such an assumption). Recently, Sataev pursued these last two results to the codimension one nontransitive case \cite{s}.
Important examples which are not sectional-Anosov were considered by the first author \cite{me1}, \cite{me2}.

In this paper we shall give positive answer in the transitive case in any dimension
(extending so \cite{appv}, \cite{col} and, partially, \cite{s}).
More precisely, we prove that every $C^2$ transitive sectional-Anosov flow on a compact manifold
has a unique SRB measure. Furthermore, such a measure is stochastically stable under small random perturbations
(extending so Kifer \cite{k} who proved stochastic stability in the case of the geometric Lorenz attractor).
This result answers in positive to a question formulated to the first author by Viana in the specific case of the multidimensional Lorenz attractor.
Let us state our result in a precise way.

Hereafter $X$ will be a $C^1$ vector field of a compact manifold $M$ inwardly transverse to the boundary $\partial M$ (if nonempty).
Denote by $X_t$ the flow generated by $X$ and
define the {\em maximal invariant set}
$$
M(X)=\displaystyle\bigcap_{t\geq0}X_t(M).
$$
We say that $X$ is {\em transitive} if $M(X)=\omega(x)$ for some $x\in M(X)$, where
$\omega(x)$ is the {\em omega-limit set} of $x$,
$$
\omega(x)=\left\{y\in M:y=\lim_{k\to\infty}X_{t_k}(x)\mbox{ for some sequence }t_k\to\infty\right\}.
$$
We say that $\Lambda\subset M(X)$ is {\em invariant} if $X_t(\Lambda)=\Lambda$ for all $t\in\mathbb{R}$.
A compact invariant set $\Lambda$ is {\em hyperbolic} if there are
a tangent bundle decomposition $T_\Lambda M=E^s_\Lambda\oplus E^X_\Lambda\oplus E^u_\Lambda$
over $\Lambda$ as well as positive constants $K,\lambda$
and a Riemannian metric $\|\cdot\|$ on $M$ satisfying

\begin{enumerate}
\item
$E^s_x\neq0$ and $E^u_x\neq0$ for all $x\in \Lambda$;
\item
$\|DX_t(x)/E^s_x|\leq Ke^{-\lambda t}$, for every $x\in \Lambda$ and $t\geq0$;
\item
$E^X_\Lambda$ is the subbundle generated by $X$;
\item
$m(DX_t(x)/E^u_x)\geq K^{-1}e^{\lambda t}$, for every $x\in \Lambda$ and $t\geq0$, where
$m(\cdot)$ indicates the conorm operation.
\end{enumerate}
We say that $X$ is an {\em Anosov flow} if $M(X)$ is a hyperbolic set of $X$.

On the other hand, a compact invariant set $\Lambda$ is {\em sectional-hyperbolic}
if every singularity in $\Lambda$ is hyperbolic and, also, there are a decomposition $T_\Lambda M=E^s_\Lambda\oplus E^c_\Lambda$
of the tangent bundle over $\Lambda$ as well as positive constants $K,\lambda$
and a Riemannian metric $\|\cdot\|$ on $M$ satisfying
\begin{enumerate}
\item
$\|DX_t(x)/E^s_x\|\leq Ke^{-\lambda t}$ for every $x\in \Lambda$ and $t\geq 0$.
\item
$\frac{\|DX_t(x)/E^s_x\|}{m(DX_t(x)/E^c_x)}\leq Ke^{-\lambda t}$, for every $x\in\Lambda$ and
$t\geq0$.
\item
$|\det(DX_t(x)/L_x)|\geq K^{-1}e^{\lambda t}$
for every $x\in \Lambda$, $t\geq0$ and every two-dimensional subspace $L_x$ of $E^c_x$.
\end{enumerate}
We say that $X$ is a {\em sectional-Anosov flow} if $M(X)$ is a sectional-hyperbolic set.

A Borel probability measure $\mu$ of $M$ is {\em invariant} if $(X_t)_*(\mu)=\mu$ for every $t\geq0$.
If, additionally, it has positive Lyapunov exponent a.e. and, also,
absolutely continuous conditional measure on the corresponding unstable manifolds, then
we say that $\mu$ is an {\em SRB measure} of $X$
(see \cite{y} for further details).

Next we introduce some basics on
random perturbations of dynamical systems \cite{k}.
Consider the family of transition
probability measures
$P^\varepsilon(t,x,\cdot)$ on
$M$ given for every $x\in M$ and $t\in I\!\! R$ (or $t\in \mathbb{Z}_+$) and $\varepsilon>0$
small enough and define Markov chains $x^\varepsilon_t$, $ t \in  I\!\! R$ in the following way:
if $x^\varepsilon_t=x$ then $x^\varepsilon_{t+\tau}$ has probability $P^\varepsilon(\tau,x,A)$ of being in $A$.
The Markov chain $x^\varepsilon_t$ for $t\in I\!\! R$ is called a {\em small random perturbation} of a flow $X_t$
if for every continuous function $h$  on $M$, we have
$$\lim_{\varepsilon\rightarrow 0}\left| \int_MP^\varepsilon(t,x,dy)h(y)-h(X_t(x))\right| =0 \, .$$
Similarly, the Markov chain  $x^\varepsilon_n$ for $n\in \mathbb{Z}_+$ is called a small random perturbation of
a map $f$ if for every continuous function $h$ on $M$, we have
$$
\lim_{\varepsilon\rightarrow 0}\left| \int_MP^\varepsilon(n,x,dy)h(y)-h(f^n(x))\right| =0.
$$
A probability measure $\nu^\varepsilon$ on $M$ is a stationary measure for the
Markov chain $x^\varepsilon_t$ if for all Borel set $A$ and any $\tau>0$, we have
$$
\int_M\nu^\varepsilon(dx)P^\varepsilon(\tau,x,A)=  \nu^\varepsilon(A).
$$
Denote by ${\mathcal B}(M)$ the set of
borelians of $M$.
Suppose that $X$ has a unique SRB measure $\mu$.
Let $P^\varepsilon\colon I\!\! R^+\times
M\times {\mathcal B}(M)\to [0,1]$ be the
transition probability measures associated to
a fixed small random perturbation $x^\varepsilon_t$
of $X$ and
$\{\mu^\varepsilon\}_{\varepsilon>0}$ be a
family of stationary measures of $P^\varepsilon$.
We say that $\mu$ is {\em stochastically stable} if for every real number sequence
$\varepsilon_i\to 0^+$ such that
$\mu^{\varepsilon_i}\to \nu$ in the weak sense one has $\nu=\mu$.

{\em By stochastic stability under small diffussion-type random perturbations}
it is meant that we are going to use transition probabilities of the form
$$
P^\varepsilon(\tau,x,A) = \int_A p^\varepsilon(\tau,x,y) dy,
$$
where $dy$ means integration with respect to the natural Lebesgue measure of
the manifold and $p^\varepsilon(\tau,x,y) $ is a solution of the diffusion equation
$$
\frac{\partial p^\varepsilon}{\partial t} (t,x,y) = (\varepsilon L + X) p^\varepsilon (t,x,y)
$$
with $L$ being an elliptic operator. Note that the elliptic operator introduces
the posibility of collision with particle in a media (or heat equation),
that gives the random part of the Markov chains. Typical solution of this equation comes with a factor that has
Gaussian behaviour, namely, $p^\varepsilon (t,x,y) \sim \exp(\frac{-\dist(X_t(x),y)}{\varepsilon})$.

Now we state our result.

\begin{maintheorem}
\label{main}
Every $C^2$ transitive sectional-Anosov flow on a compact manifold
has a unique SRB measure which is stochastically stable
under small diffusion-type random perturbations.
\end{maintheorem}

This result (announced twice in \cite{mm} and \cite{mmm}) extend the recent paper by Leplaideur and Yang \cite{ly}.

\section{Preliminaries}

\noindent
This section is to give the results needed to prove
2.1 in \cite{k}, which is essentially a linear version of the perturbation and only need hyperbolic behaviour along orbits.
This can be achieved if the orbits remains outside the singularities as it is shown in the following propositions and lemmas.

Hereafter $X$ will be a $C^2$ transitive sectional-Anosov flow of a compact manifold $M$.

\begin{lemma}[Shadowing Lemma]\label{ShadLemSH}
There exists a constant $C$ such that if $x_0,\dots,x_n$ is a $\delta$-pseudo-orbit of $F=X_\tau$ satisfying
\begin{equation}
\min_{0\leq i\leq n}\dist(x_i,Sing(X))>C\delta,
\end{equation}
then one can find a point $y\in M$ such that
\begin{equation}
\max_{0\leq i\leq n}\dist(x_i,F^iy)\leq Cn\delta
\end{equation}
\end{lemma}

\begin{prop}
\label{SHtoH1}
There are 
positive constants $\gamma>0$ and $\alpha>0$, not depending on $y\in M$, such that
for any $y\in M$ that shadows a $\delta$-pseudo-orbit as above there is an invariant splitting 
$$
T_{F^ly}M=\tilde H_{F^ly}\oplus H^u_{F^ly},
\,\,\,\,\,\,\,\,\,\,l\in I\!\! N,
$$
satisfying
$$
\angle(\tilde H_{F^ly}, H^u_{F^ly})>\alpha
\,\,\,\,\,\,\,\mbox{ and }
\,\,\,\,\,\,\,\,
||DF^{-l}\zeta||\leq \gamma^{-1}e^{-\gamma l}||\zeta||,
\,\,\,\,\,\,\,\,\,\,\forall l\in I\!\! N.
$$
\end{prop}

Let $J(x)$ be the absolute value of the Jacobian of the derivative of $F$ at $x$, and $J_n(x)$
the absolute value of the Jacobian of the $F^n$ at $x$. Define also 
$d_n(x,y)=
\max\{\dist(F^kx,F^ky), 0\leq |k|\leq |n|\mbox{\ \ and \ \ }kn\geq 0\}$
and $K_\rho(x,n)=\{y\ :\ d_n(x,y)\leq \rho\}$.
 
\begin{prop}[Volume Lemma]\label{VoluLemSH}
Then there exists $\tilde \rho,C_\rho, C>0$ such that for
any positive $\rho\leq \tilde\rho,n\geq 0$ and $x\in M(X)$
\begin{equation}\label{vollema3.13}
C^{-1}_\rho\leq m(K_\rho(x,n) J_n(x)\leq C_\rho
\end{equation}
where $m$ is the Riemannian volume, and for each $y\in K_\rho(x,n)$
\begin{equation}\label{vollema3.14}
C^{-1}\leq J_n(x,n) (J_n(y))^{-1}\leq C
\end{equation}
If $X$ has no singularities (hence Anosov) and $F=X_\tau$ then (\ref{vollema3.13}) remains true and (\ref{vollema3.14}) must be replaced by 
\begin{equation}
C^{-1}\leq J_n(F^ux) (J_n(y))^{-1}\leq C
\end{equation}
with $|u|\leq c\rho$, where $c\geq 0$ depends only on $X$.
\end{prop}

\section{Proof of Theorem \ref{main}}

\noindent
Let $X$ be a $C^2$ transitive sectional-Anosov flow of a compact manifold $M$.
Denote by $Sing(X)$ the set of
singularities of $X$.
If $Sing(X)=\emptyset$
then $X$ is Anosov
and then
the result follows from classical results \cite{k}.
Then, we can assume that
$X$ is a genuine sectional-Anosov flow, i.e.,
$Sing(X)\neq\emptyset$.
We keep in mind the notation concerning small
random perturbations as in the Introduction.

The proof of Theorem \ref{main}
is based on following lemmas to be proved in the final sections.

\begin{lemma}
\label{l1}
Given $\epsilon>0$ small there is $C>0$ such that
if $\gamma>0$ is small, then
$\forall x\in M$, $\forall k\geq \log(\varepsilon^{-m})$
one has
$$
P^\varepsilon(k,x,B_{\varepsilon^{1-\gamma}}(Sing(X)))
\leq C\varepsilon^\gamma
$$
\end{lemma}

This means that the probability that a Markov chain arrives too close to the
singularities is very small, while the next one means that for those which do not get close to the
singularities we have the absolutely continuous property.

Recall that
${\mathcal B}(M)$ denotes the set of borelians
of $M$.
We denote by $[E,D]$ a rectangle
consisting of points $[x,y]=W^s_\eta(x)\pitchfork W_\rho^u(y)$, for local invariant manifolds, with
small enough $\eta$ and $\rho$ (see \cite{k} p. 142-143).
 
\begin{lemma}
\label{l2}
For every $\varepsilon>0$ small
there is $C>0$
such that
$\forall 1>\gamma>0$,
$\forall x\in M$ and
$\forall Q\in {\mathcal B}(M)$, with $\dist(Q,Sing(X))>0$,
of the form $Q=[E,D]$
one has
$$
I^\varepsilon_0(\varepsilon^{1-\gamma},n,x,Q)
\leq C mes^u(E) +O(\varepsilon),
$$
where:
\begin{itemize}
\item
$mes^u$ is the Lebesgue measure
in the unstable direction;
\item
$
I_0^\varepsilon(\rho,n,x,\Gamma)=
P^\varepsilon_x\{
\min_{0\leq k\leq n-1}dis(x^\varepsilon_k,Sing(\Lambda))\geq \rho
\mbox{ and }x^\varepsilon_n\in \Gamma\}$;
\item
$x^\varepsilon_k$ is the Markov chain
induced by the time $\tau$-map $X_{\tau}$;
\item
$P^\varepsilon_x\{x^\varepsilon_n\in \Gamma\}=
P^\varepsilon(n,x,\Gamma)$.
\item $n\geq(\log(\varepsilon))^2$
\end{itemize}
\end{lemma}

With these lemmas we
are ready to prove our theorem.

We know by definition that
$$
\mu^\varepsilon(\Gamma)=
\int_MP^\varepsilon(n,x,\Gamma)d\mu^\varepsilon(x),
$$
for every $n$, $x$ and $\Gamma$.
We shall need the inequality
\begin{equation}
\label{*}
P^\varepsilon(n,x,[E,D])
\leq C\cdot mes^u(E)+O(\varepsilon)
\end{equation}
because it implies
$$
\mu^\varepsilon([E,D])
\leq
C\cdot mes^u(E)+O(\varepsilon).
$$
If $\epsilon_i\to 0^+$ as $i\to\infty$
and $\mu^{\varepsilon_i}\to \mu^*$
then we have
$$
\mu^*([E,D])
\leq C\cdot mes^u(E)
$$ which proves that
$\mu^*$ is absolutely continuous
in the unstable direction and supported
in $M(X)$.
Since $M(X)$ is sectional-hyperbolic we have
that every point of
$M(X)$ has at least one positive Lyapunov exponent,
so $\mu^*$ is a SRB measure of $X$.
This argument shows that $M(X)$ supports
SRB measures.
Once we prove that there is only one SRB measure
in $M(X)$
we simultaneously prove the desired
stochastic stability.
So, the proof of Theorem
\ref{main} needs the following
claim:

\begin{claim}
\label{cl1}
$M(X)$ supports a unique SRB measure.
\end{claim}

\begin{proof}
Since $\mu^*$ is absolutely continuous
with respect to $mes^u$
and $M(X)$ is transitive,
we obtain that
$\mu^*$ is positive in open sets in the unstable direction
so it is equivalent to $mes^u$.
If $\nu^*$ were another SRB measure
supported in $M(X)$ then
$\nu^*$ is absolutely continuous with respect to
$mes^u$ so it is also absolutely continuous
with respect to $\mu^*$.
But $\nu^*$ is ergodic as it is SRB
so $\nu^*$ and $\mu^*$ are the same.
This proves the claim.
\end{proof}

Now we turn on to the proof of
(\ref{*}).
This is merely a computation using
lemmas \ref{l1} and \ref{l2}:
For all $n$, $\rho$ and $\Gamma$ as above one has
$$
P^\varepsilon(n,x,\Gamma)=
P^\varepsilon_x\{x_n^\varepsilon\in \Gamma\}=
P^\varepsilon_x\{
\min_{0\leq k\leq n-1}dist(x^\varepsilon_k,Sing(X))\geq \rho
\mbox{ and }x^\varepsilon_k\in \Gamma\}$$
$$+ P^\varepsilon_x\{
\exists k:dist(x^\varepsilon_n,Sing(X))<\rho, \mbox{\ and\ } x^\varepsilon_n\in\Gamma\}
$$
$$
=I^\varepsilon_0(\rho,n,x,\Gamma)
+
P^\varepsilon_x\{
\exists k:dist(x^\varepsilon_k,Sing(X)<\rho, \mbox{\ and\ } x^\varepsilon_n\in\Gamma\}.
$$
Replacing $\Gamma=[E,D]$ and taking
$n\in [(\log\varepsilon)^{2},(\log\varepsilon)^{4}]$
and $\rho=\varepsilon^{1-\gamma}$
we obtain by lemmas \ref{l1}-\ref{l2} that
$$
P^\varepsilon(n,x,[E,D])\leq
C\cdot mes^u(E)+O(\varepsilon)
$$
proving (\ref{*}).

\section{Proof of Lemma \ref{l2}}

\noindent
In the flow case(\footnote{We are in the diffeomorphism case $F=X_\tau$ for some $\tau>1$}),
set $\tilde W_\rho^s(\Gamma)=\bigcup_{|t|\leq\rho}W_\rho^s(\Gamma)$. For diffeomorphisms, set $\tilde W_\rho^s(\Gamma)=W_\rho^s(\Gamma)$.

It is enought to prove the lemma for the case where $E=W_\eta^u(z)$ and $D=W_\rho^s(z)\cap M(X)$, for all $z\in M(X)$ and $\rho,\eta>0$.

Choose $v_i\in E$ for $i=1,\ldots,k_\varepsilon$ such that
$$
E\subset \bigcap_{i}W_\varepsilon^u(v_i)\mbox{\qquad and \qquad} \sum_{i}m^u(W_\varepsilon^u(v_i))\leq 3^m m^u(E).
$$

Denote 

$$
I_1(\rho,\delta,n,x,\Gamma)=P_x^\varepsilon\left\{\begin{array}{l}\min_{0\leq k\leq n-1}
dis(x^\varepsilon_k,Sing(X))\geq \rho,\\ \dist(Fx_i^\varepsilon,x_{i+1}^\varepsilon)<\delta,
i=0,\ldots,n-1\mbox{\ \ and\ \ } \\ x_n^\varepsilon\in\Gamma\end{array}\right\}
$$

Then for $\varepsilon$ small enough
$$
I_0^\varepsilon(\rho,n(\varepsilon),x,\tilde W_\rho^s(E))\leq I_1(\rho,\delta(\varepsilon),n(\varepsilon),x,
\tilde W_\rho^s(E)\cap W_{\varepsilon^{1-2\beta}}^s(M(X)))+
{\mathcal O}(\varepsilon).
$$
where $\delta(\varepsilon)$ is chosen so that we can approximate
transition probabilities of Markov chains with transition probabilities of Markov chains that are
also $\delta(\varepsilon)$-pseudo-orbits. If we choose $\delta(\varepsilon)=\varepsilon^{1-\beta}$
with $0<\beta<\alpha$ and small (to be chosen later) the error is of the order of $\exp(-\beta/3)$. Also,
to make this approximation we need $n(\varepsilon)>(\log(\varepsilon)^2$.
But,
$$I_1(\rho,\delta(\varepsilon),n(\varepsilon),x,
\tilde W_\rho^s(E)\cap W_{\varepsilon^{1-2\beta}}^s(\Lambda))\leq \sum_{i=1}^{k_\varepsilon}I_1^\varepsilon(\rho,\delta(\varepsilon),n(\varepsilon),x,
A_i^\varepsilon),
$$
where $A_i^\varepsilon=\tilde W_\rho^s(E_i^\varepsilon)$, $E_i^\varepsilon=W_\varepsilon^u(v_i)$.

From the definition, for each $x\in W_{\varepsilon^{1-2\beta}}^s(M(X))$ there exists $\tilde x$ such that
$x\in W_{\varepsilon^{1-2\beta}}^s(\tilde x)$ and $\tilde x\in M(X)$.

For each $\tilde x$, take $G_j^\varepsilon=\overline{W_{j\varepsilon}^u(\tilde x)
\backslash W_{(j-1)\varepsilon}^u(\tilde x)}$, $j=1,2,\ldots,\lfloor\varepsilon^{-4\beta}\rfloor+1$.
So, for each $i,j$ we have that $G_j^\varepsilon\cap F^{n(\varepsilon)}W_{2\rho}^s(v_i)$ consists of points $z_{ijk}$ for $k=1,\ldots,k^{ij}$.

If $w=(x,y,\ldots,y_n)$ is a $\delta(\varepsilon)$-pseudo-orbit, with $x\in W_{\varepsilon^{1-2\beta}}^s(\tilde x)$,
$\tilde x\in \Lambda$, $y_n\in A_i^\varepsilon$ then there exists $y^w$ such that $\dist(y_l,F^ly^w)\leq Cn\varepsilon^{1-2\beta}$,
for all $l=10,\ldots, n$, using the Shadowing Lemma (see Lemma \ref{ShadLemSH}),
where $y^w\in \tilde W_{\varepsilon^{1-4\beta}}^s(W_{\varepsilon^{1-4\beta}}^u(\tilde x))$
and $F^{n(\varepsilon)}y^w\in \tilde W_{\varepsilon^{1-4\beta}}^s(W_{\varepsilon^{1-4\beta}}^u(A_i^\varepsilon))$.

That is, there exists $i,j,k$ such that $j\leq \lfloor \varepsilon^{-4\beta}\rfloor+1$ and
$ \dist(F^ly^w,F_ lz_{ijk})\leq \frac{1}{2}\varepsilon^{1-5\beta}$ for all $l=0,\ldots,n(\varepsilon)$.

Then, $\dist (y_l,F^lz_{ijk})\leq \varepsilon^{1-5\beta}$  and
\begin{equation}\label{desigI2}
I_1^\varepsilon(\rho,\delta(\varepsilon),n(\varepsilon),x,
A_i^\varepsilon)\leq \sum_{j\leq\lfloor\varepsilon^{-4\beta}+1\rfloor}I_2^\varepsilon(\varepsilon^{1-5\beta},
\delta(\varepsilon),n(\varepsilon),z_{ijk},A_i^\varepsilon)
\end{equation}
where
\begin{eqnarray*}
I_2^\varepsilon(\rho,\delta,n,z,\Gamma)&=&P_x^\varepsilon
\left\{\begin{array}{l}\min_{0\leq k\leq n-1}dis(x^\varepsilon_k,Sing(\Lambda))\geq \rho\\
\dist(x_l^\varepsilon,F^lz)\leq \delta\\ x_n^\varepsilon\in\Gamma\end{array}\right\}\\
&=&\int_{U_\rho(Fz)}\cdots \int _{U_\rho(F^{n-1}z)}\int _{U_\rho(F^{n}z)\cap\Gamma}
q_{Fx}^\varepsilon(y_1)\cdots q_{Fy_{n-1}}^\varepsilon(y_n)\\ && \qquad\qquad \qquad \qquad\qquad \qquad \cdots dm(y_1) \cdots dm(y_n)\\
\end{eqnarray*}
That is, for every $z=z_{ijk}$ we have that $I_2$ in the sum of equation (\ref{desigI2}) is less or equal 
\begin{eqnarray*}
& \leq & 
(1+\varepsilon^\alpha)
\int_{U_{\varepsilon^{1-4\beta}}(Fz)}\cdots \int _{U_{\varepsilon^{1-4\beta}}(F^{n-1}z)}\int _{U_{\varepsilon^{1-4\beta}}(F^{n}z)\cap A_i^\varepsilon}\\
& & 
\varepsilon^{-m}r_{Fx}(\frac{1}{\varepsilon}\exp_{Fx}^{-1}(y_1))\cdots
\varepsilon^{-m}r_{Fy_{n-}}(\frac{1}{\varepsilon}\exp_{Fy_{n-1}}^{-1}(y_n))
dy_1\cdots dy_n
\end{eqnarray*}

After this preparation we can lift the problem to the tangent bundle in the same way
as in Theorem 4.1 of Kifer \cite{k} which essentially uses the Volumen Lemma \ref{VoluLemSH} and Theorem 2.1 and 3.10 of \cite{k}.
Observe that we can use Theorem 2.1 of  \cite{k} because for pseudo-orbits not aproaching the set $Sing(\Lambda)$  our transformations
behaves like a hyperbolic one, see propositions \ref{SHtoH1} and \ref{VoluLemSH}.\hfill$\square$

\section{Proof of Lemma \ref{l1}}

\noindent
By the Chapman-Kolmogorov formula for any $l<k$ one has
\begin{eqnarray}P^\varepsilon(k,x,B_{\varepsilon^{1-\gamma}}(Sing(X))&=&
\int_M P^\varepsilon(k-l,x,dz)P^\varepsilon(l,z,B_{\varepsilon^{1-\gamma}}(Sing(X))\\
&\leq&\nonumber \sup_{z\in M}P^\varepsilon(l,z,B_{\varepsilon^{1-\gamma}}(Sing(X))
\end{eqnarray}
so if the conclusion of Lemma \ref{l1} is true for $k=l$ it remains true for any $k\geq l$.

Take \begin{equation}D=\sup_{x\in M}||F'(x)||<\infty.\end{equation}
Define 
\begin{equation}\label{des1.33}
l=\left[\frac{1}{2}\beta(\log(D+1))^{-1}\log(\frac{1}{\varepsilon})\right]
\end{equation}
where $1>\beta>0$ and $[\ \cdot\ ]$ means the integral part. It is not dificult to show that with this definition we have
\begin{equation}\label{des1.34}
(D+1)^{l+1}\leq e^{-\beta/2}
\end{equation}
We also use $\beta$ to aproximate the probability with $\varepsilon^{1-\beta}$-pseudo-orbits using Lemma 1.1 of
\cite{k} (p. 101) and obtain the following inequality
\begin{equation}
P^\varepsilon(l,x,\Gamma)\leq I_1(\varepsilon^{1-\beta},l,\Gamma)+Cl\varepsilon^{-2m}mes(\Gamma)\exp(-\frac{\alpha\varepsilon^{1-\beta}}{2\varepsilon})$$ 
$$\leq I_1(\varepsilon^{1-\beta},l,\Gamma)+\exp(-\frac{\alpha}{3\varepsilon^\beta})
\end{equation}
Where $$I_1(\varepsilon^{1-\beta},l,x,\Gamma)=$$ $$=P^\varepsilon\{(dist(F(X_i^\varepsilon),X_{i+1}^\varepsilon)<\varepsilon^{1-\beta}
\mbox{\ \ for all\ \ }i=1,\ldots,l-1\mbox{\ \ and\ \ }X^\varepsilon_l\in \Gamma\}$$

From the continuity of $F$ and (\ref{des1.34}) every $\varepsilon^{1-\beta}$-pseudo-orbit $y_0=x,y_1,\ldots,y_l$ satisfies
\begin{equation}
\label{des1.35}dist(F^i(x),y_i)<\varepsilon^{1-2\beta}\mbox{\ \ for all\ \ } i=0,\ldots,l.
\end{equation}
So, we can write
\begin{equation}\label{des1.37}
I_1(\varepsilon^{1-\beta},l,x,\Gamma)\leq I_3^\varepsilon(l,x,\Gamma)
\end{equation}
where $2\beta<\alpha$, and
\begin{eqnarray}\hskip20pt
 I_3^\varepsilon(l,x,\Gamma)&=&P_x^\varepsilon\{X_i^\varepsilon\in U^{(i)}\mbox{\ for all\ }
i=1,\ldots,l\mbox{\ and\ } X_l^\varepsilon\in\Gamma\}\\
\nonumber&=&\int_{U^{(0)}}\hskip-6pt\ldots\int_{U^{(l-1)}}\hskip-2pt
\int_{U^{(l)}\cap\Gamma}\hskip-4pt q_{Fx}^\varepsilon(y_1)q_{Fy_1}^\varepsilon(y_2)\ldots q_{Fy_{l-1}}^\varepsilon(y_l)dy_1\ldots dy_l \\
\nonumber &\leq & (1+\varepsilon^\alpha)^l
\int_{U^{(0)}}\hskip-6pt\ldots\int_{U^{(l-1)}}\hskip-2pt
\int_{U^{(l)}\cap\Gamma}\hskip-4pt
\varepsilon^{-m} r_{Fx}(\frac{\sigma(Fx,y_1)}{\varepsilon})\\
&&\times\ \ \varepsilon^{-m} r_{Fy_1}(\frac{\sigma(Fy_1,y_2)}{\varepsilon})\ldots
\varepsilon^{-m} r_{Fy_{l-1}}(\frac{\sigma(F_{l-1},y_l)}{\varepsilon})dy_1\ldots dy_l.
\end{eqnarray}
where $U^{(i)}=\{v:dist(v,z_i)<\varepsilon^{1-2\beta}\}$. Since of our choice of the points $z_i$, we can have
$$||Fy_i-y_{i+1}||=||Fy_i-z_{i+1}+z_{i+1}-y_{i+1}||\leq ||Fy_i-z_{i+1}||+|F'(z_i)(z_i-y_i)|+ord(\varepsilon^{2-5\beta})$$
provided $\beta<\frac{1}{5}$

This will lead to an expresion that can be bounded by
\begin{eqnarray}
I_4^\varepsilon(l,x,\Gamma)&=&\int_{\re^n}\ldots\int_{\re^n}\int_{\Gamma}
\varepsilon^{-m}r_{z_1}(\frac{\eta_1}{\varepsilon})\varepsilon^{-m}
r_{z_1}(\frac{\eta_2-F'z_1(y_i-z_i)}{\varepsilon})\ldots\\
&&\times\ \  \varepsilon^{-m}r_{z_l}(\frac{\eta_l-Fz_{l-1}(y_{l-1}-z_{l-1})}{\varepsilon})
d\eta_1\ldots d\eta_l
\end{eqnarray}

From here we con follow the same calculation used in the proof of Theorem II.2.1 of \cite{k} and obtain
\begin{equation}
I_4^\varepsilon(l,x,\Gamma)\leq \tilde C\varepsilon^{-m}\lambda^{-l} mes({\Gamma})
\end{equation}
where $\lambda$ is the expansion rate of the volume. Now taking $U_{\varepsilon^{1-\gamma}}(Sing(\Lambda))$ and $l$ as in (\ref{des1.33})
we arrive to the conclusion of the lemma.\hfill$\square$

\end{document}